\documentclass[reqno, a4paper]{elsarticle}
\usepackage{amssymb,amscd,amsthm,latexsym,qsymbols, mathrsfs }
\usepackage[all]{xy}
\CompileMatrices

\CompileMatrices

\newtheorem{theorem}{Theorem}[section]
\newtheorem{lemma}[theorem]{Lemma}
\newtheorem{proposition}[theorem]{Proposition}

\newtheorem*{property-P}{Property P}

\theoremstyle{definition}

\newtheorem{definition}[theorem]{Definition}

\newtheorem{remark}[theorem]{Remark}

\renewcommand{\to}{\rightarrow}

\renewcommand{\Im}{\ensuremath{\mathsf{Im\,}}}

\newcommand{\coker}{\ensuremath{\mathsf{coker\,}}}

\renewcommand{\ker}{\ensuremath{\mathsf{ker\,}}}
\newcommand{\Ker}{\ensuremath{\mathsf{Ker\,}}}

\newcommand{\Cc}{\ensuremath{\mathbb{C}}}

\newcommand{\Mf}{\mathcal M}
\newcommand{\Ef}{\mathcal E}

\newcommand{\Ab}{\ensuremath{\mathsf{Ab}}}

\newcommand{\Xc}{\ensuremath{\mathbb{X}}}
\newcommand{\Yc}{\ensuremath{\mathbb{Y}}}
\newcommand{\GrH}{\ensuremath{\mathsf{Grp(Haus)}}}
\newcommand{\GrTot}{\ensuremath{\mathsf{Grp(TotDis)}}}

\newcommand{\Gp}{\ensuremath{\mathsf{Grp}}}

\newcommand{\Ec}{\ensuremath{\mathcal{E}}}

\renewcommand{\hom}{\ensuremath{\mathrm{Hom}}}

\newbox\skewpullbackbox
\setbox\skewpullbackbox=\hbox{\xy 0;<1mm,0mm>: \POS(4,0)\ar@{-} (-4,0) \ar@{-} (8,4)
\endxy}

\newbox\ksewpullbackbox
\setbox\ksewpullbackbox=\hbox{\xy 0;<1mm,0mm>: \POS(0,-8)\ar@{-} (0,-4) \ar@{-} (4,-4)
\endxy}

\newbox\pullbackbox
\setbox\pullbackbox=\hbox{\xy 0;<1mm,0mm>: \POS(4,0)\ar@{-} (0,0) \ar@{-} (4,4)
\endxy}
\def\pullback{\copy\pullbackbox}

\newbox\pushoutbox
\setbox\pushoutbox=\hbox{\xy 0;<1mm,0mm>: \POS(0,4)\ar@{-} (0,0) \ar@{-} (4,4)
\endxy}

\begin{document}

\title{Monotone-light factorisation systems and \\ torsion theories}

\author[MG]{Marino Gran\fnref{fn1}}
\ead{marino.gran@uclouvain.be}

\author[TE]{Tomas Everaert}
\ead{teveraer@vub.ac.be}

\address[MG]{Universit\'e catholique de Louvain,
Institut de Recherche en Math\'ematique et Physique,
Chemin du Cyclotron 2, bte.
L7.01.02,
1348 Louvain-la-Neuve, Belgium.}
\address[TE]{Vakgroep Wiskunde,
Vrije Universiteit Brussel,
 Department of Mathematics,
 Pleinlaan 2,
1050 Brussel,
 Belgium.
 }

\begin{abstract}
Given a torsion theory $(\Yc,\Xc)$ in an abelian category $\Cc$, the reflector $I \colon \Cc \rightarrow \Xc$ to the torsion-free subcategory $\Xc$ induces a reflective factorisation system $(\Ec, \Mf)$ on $\Cc$.
It was shown by A. Carboni, G.M. Kelly, G. Janelidze and R. Par\'e that $(\Ec, \Mf)$ induces a monotone-light factorisation system $(\Ef',\Mf^*)$ by simultaneously stabilising $\Ec$ and localising $\Mf$, whenever the torsion theory is hereditary and any object in $\Cc$ is a quotient of an object in $\Xc$. We extend this result to arbitrary normal categories, and improve it also in the abelian case, where the heredity assumption on the torsion theory turns out to be redundant. Several new examples of torsion theories where this result applies are then considered in the categories of abelian groups, groups, topological groups, commutative rings, and crossed modules. \\ \vspace{1mm}

  \noindent {\bf Keywords}: factorisation system, monotone-light factorisation, torsion theory, normal category, torsion-free groups, totally disconnected groups, reduced rings.\\ \vspace{1mm}
  
  \noindent Math. Subj. Class.: 18E40, 18A40, 16S90, 16N80, 20K15, 54H11.
\end{abstract}

\maketitle

\section*{Introduction}

Let $\Cc$ be a category. Recall that a \emph{prefactorisation system} on $\Cc$ is given by two classes $\Ef$ and $\Mf$ of morphisms such that  $\Ef=\Mf^{\uparrow}$ and $\Mf=\Ef^{\downarrow}$, where 
\[
\Ef^{\downarrow}=\{m | e\downarrow m \  \textrm{for all} \ e\in \Ef\}, \  \ \  \Mf^{\uparrow}=\{e | e\downarrow m \  \textrm{for all} \ m\in \Mf\}
\]  
and the relation $\downarrow$ is defined as follows: for morphisms $e$ and $m$ in $\Cc$,  one writes $e\downarrow m$ if there exists, for every commutative square
\[
\xymatrix{
A \ar[r]^e \ar[d]_a & B \ar@{.>}[ld]|{d}   \ar[d]^b\\
C \ar[r]_m & D,}
\]
a unique morphism $d$ such that $d\circ e=a$ and $m\circ d=b$. A prefactorisation system $(\Ef,\Mf)$ is a \emph{factorisation system} if for every morphism $f$ in $\Cc$ there exist morphisms $e\in \Ef$ and $m\in \Mf$ such that $f=m\circ e$.

From \cite{CHK} we know that one can associate with any full reflection
 \begin{equation}\label{reflection}
 \xymatrix@=30pt{
{\Cc \, } \ar@<1ex>[r]_-{^{\perp}}^-{I} & {\, \Xc \, }
\ar@<1ex>[l]^H  }
 \end{equation}
a prefactorisation system $(\Ef,\Mf)$ (on $\Cc$) where $\Ef$ consists of all morphisms inverted by the reflector $I$ and $\Mf$ is the closure under pullback (along arbitrary morphisms in $\Cc$) of the class of all morphisms which lie in the image of the inclusion functor $H$. In fact, $(\Ef,\Mf)$ often is a factorisation system, and this is the case in particular when $I$ is \emph{semi-left-exact}: $I$ preserves those pullbacks of the form
\[
\xymatrix{
P \ar@{}[rd]|<<{\pullback}\ar[d] \ar[r]& B \ar[d]^{\eta_B}\\
H(X) \ar[r] & HI(B)}
\]
where $\eta_B$ is the reflection unit. In this case, if $f\colon A\to B$ is a morphism in $\Cc$ and if
\[
\xymatrix{
A \ar[rrd]^f \ar[rd] \ar[ddr] && \\
& P  \ar@{}[rd]|<<{\pullback} \ar[r]\ar[d] & B\ar[d]^{\eta_B} \\
&HI(A) \ar[r]_-{HI(f)} & HI(B)
}
\]
is the commutative diagram obtained by pulling back $HI(f)$ along $\eta_B$, then the factorisation $A\to P$ to the pullback $P=HI(A)\times_{HI(B)}B$ is sent by $I$ to an isomorphism, whence its $(\Ef,\Mf)$-factorisation is given by $A\to P\to B$. Unless specified otherwise, throughout the article $(\Ef,\Mf)$ will always denote a factorisation system associated with some semi-left-exact reflection (\ref{reflection}).


Recall from \cite{CJKP} that a reflection is semi-left-exact if and only if it is \emph{admissible} in the sense of categorical Galois theory \cite{J}, where the morphisms in $\Mf$ are called \emph{trivial coverings}. A fundamental observation in Galois theory is that, for a morphism, being a trivial covering is in general not a ``local'' property, by which we mean the following: if  $f\colon A\to B$ is a morphism, then $f$ is said to possess a certain property \emph{locally} if there exists an effective descent morphism $p\colon E\to B$ (i.e. a morphism $p\colon E\to B$ such that the pullback functor $p^*\colon (\Cc\downarrow B)\to (\Cc\downarrow E)$ is monadic \cite{JST}) such that the pullback $p^*(f)\colon E\times_B A\to E$ has this property. Using the fact that the functor $p^*$ reflects isomorphisms, one can verify that any morphism which is locally a member of the class $\Ef$ is already itself in $\Ef$, even for an arbitrary factorisation system $(\Ef,\Mf)$.  In this sense we can say that, for a morphism, being a member of $\Ef$  is a ``local'' property. However, as noted above, the property of being in $\Mf$ is \emph{not} a ``local'' one in this sense: for an arbitrary semi-left-exact reflection there may exist morphisms which are locally trivial coverings (these are exactly the \emph{coverings} of categorical Galois theory) but which are not trivial coverings themselves. Let us denote the class of morphisms that are locally in $\Mf$ by $\Mf^*$.

``Dually'', we have that the class $\Mf$ is always stable under pullback (even for an arbitrary (pre)factorisation system $(\Ef,\Mf)$) while for $\Ef$ this is generally false: it was proved in \cite{CHK} that $\Ef$ is pullback-stable if, and only if, $I$ preserves \emph{arbitrary} pullbacks (i.e. if it is a \emph{localisation}). We shall write $\Ef'$ for the class of morphisms ``stably'' in $\Ef$, i.e. those morphisms of which every pullback is in $\Ef$.

Here, we shall be interested in the following question: if we replace $\Ef$ by $\Ef'$ and $\Mf$ by $\Mf^*$, do we obtain again a factorisation system $(\Ef',\Mf^*)$, which is then both ``stable'' and ``local''? As explained in \cite{CJKP} the answer in general is \emph{no}, even though we always have that  $\Ef'\subseteq (\Mf^*)^{\uparrow}$. However, there is a number of important examples for which the answer is \emph{yes}. For instance, this is the case for the factorisation system  $(\Ef,\Mf)$  on the category of compact Hausdorff spaces associated with the reflective subcategory of totally disconnected spaces: in this case,  $(\Ef',\Mf^*)$ is the Eilenberg/Whyburn monotone-light factorisation system for maps of compact Hausdorff spaces \cite{Eilenberg,Whyburn}. For this reason we shall call a factorisation system \emph{monotone-light} if it is of the form $(\Ef',\Mf^*)$ for some $(\Ef,\Mf)$.

Another example from \cite{CJKP}, or in fact a class of examples, is given by any hereditary torsion theory in an abelian category $\Cc$ with the property that every object of $\Cc$ is the quotient of an object in the ``torsion-free'' subcategory (see below). Also in this case the factorisation system $(\Ef,\Mf)$ associated with the ``torsion-free'' reflection induces a monotone-light factorisation system  $(\Ef',\Mf^*)$. 

In view of the recent interest in torsion theories in contexts more general than the one of abelian categories (see for instance \cite{BG, CDT, EGinf,GR,JT}), a natural question to ask is whether torsion theories induce monotone-light factorisation systems also in a non-abelian context. The aim of the present article is to show that this is indeed the case for torsion theories in a \emph{normal} category \cite{ZJ} (as, for instance, any semi-abelian category \cite{JMT}). This allows us to study several new monotone-light factorisation systems in the categories of groups, topological groups, commutative rings and crossed modules, for example.
Moreover, the torsion theory is not even required to be hereditary, but only to satisfy a suitable property (see condition $(N)$ below) which always holds in the abelian case. Hence, our result will improve the one concerning torsion theories in \cite{CJKP}, even within the abelian context.

\section{Main Results}

By a \emph{pointed} category we mean, as usual, a category $\Cc$ which admits a \emph{zero-object}, i.e.~an object $0\in\Cc$ which is both initial and terminal. For any pair of objects $A, B\in\Cc$, the unique morphism $A\to B$ factorising through the zero-object will also be denoted by $0$. A \emph{short exact sequence} in $\Cc$ is given by a composable pair of morphisms $(k,f)$, as in the diagram
\begin{equation}\label{ses}
\xymatrix{
0 \ar[r] & K \ar[r]^k & A \ar[r]^f \ar[r] & B\ar[r] & 0,}
\end{equation}
such that $k=\ker (f)$ is the kernel of $f$ (the pullback along $f$ of the unique morphism $0\to B$) and $f=\coker (k)$ is the cokernel of $k$ (the pushout by $k$ of $K\to 0$). Given such a short exact sequence, we shall sometimes denote the object $B$ by $A/K$.

\begin{definition}
Let $\Cc$ be a pointed category. A pair $(\Yc,\Xc)$ of full and replete subcategories of $\Cc$ is called a \emph{torsion theory} in $\Cc$ if the following two conditions are satisfied: 
\begin{itemize}
\item
$\hom_{\Cc}(Y,X)=\{0\}$ for any $X\in\Xc$ and $Y\in\Yc$;
\item
for any object $A\in\Cc$ there exists a short exact sequence
\begin{equation}\label{torsionses}
0\to Y \to A \to X \to 0
\end{equation}
such that $X\in\Xc$ and $Y\in\Yc$. 
\end{itemize}
\end{definition}
$\Yc$ is called the \emph{torsion part} and $\Xc$ the \emph{torsion-free part} of the torsion theory $(\Yc,\Xc)$. A full and replete subcategory $\Xc$ of a pointed category $\Cc$ is \emph{torsion-free} if it is the torsion-free part of some torsion theory in $\Cc$. \emph{Torsion} subcategories are defined dually. An important property we shall need later on is that both the torsion and torsion-free parts are closed in $\Cc$ under extensions \cite{JT}: for every short exact sequence (\ref{ses}), if $K$ and $B$ are in $\Xc$ (respectively, in $\Yc$) then also $A$ is in $\Xc$ (respectively, in $\Yc$).
The terminology comes from the classical example  of the torsion theory $( \Ab_{t.}, \Ab_{t.f.})$ in the variety $\Ab$ of abelian groups, where 
$\Ab_{t.f.}$ consists of all torsion-free abelian groups in the usual sense (=abelian groups satisfying, for every $n\geq 1$, the implication $nx=0 \Rightarrow x=0$) and $\Ab_{t.}$ consists of all torsion abelian groups.

A torsion-free subcategory is necessarily a reflective subcategory, while a torsion subcategory is always coreflective: the reflection and coreflection of an object $A$ are given by the short exact sequence (\ref{torsionses}), which is uniquely determined, up to isomorphism. Let $\Xc$ be a reflective subcategory of a pointed category $\Cc$ with pullback-stable normal epimorphisms such that each unit $\eta_A \colon A \rightarrow HI(A)$ is a normal epimorphism  (=the cokernel of some morphism). Then $\Xc$ is torsion-free if, and only if, it is semi-left-exact (see Theorem 1.6 in \cite{EGinf}). Hence, in this context, any torsion-free subcategory induces a reflective factorisation system $(\Ef,\Mf)$, with  $\Ef$ the class of morphisms that are inverted by the reflector $I \colon \Cc \rightarrow \Xc$ and $\Mf$ the closure under pullback of the class of all morphisms in the image of the inclusion functor $H$. Now, replacing $\Ef$ by the class $\Ef'$ of morphisms stably in $\Ef$, and $\Mf$ by the class $\Mf^*$ of morphisms locally in $\Mf$, we would like $(\Ef',\Mf^*)$ to be a factorisation system too. From \cite{CJKP} we know that this is the case whenever the category $\Cc$ is abelian, and if, moreover, the following two conditions are satisfied: the torsion theory $(\Yc,\Xc)$ is hereditary (which means that the torsion part $\Yc$ is closed in $\Cc$ under subobjects) and $\Xc$ ``covers'' $\Cc$ in the following sense: for any object $B\in\Cc$, there exists an epimorphism $E\to B$ such that $E\in\Xc$.  

Here, instead of asking $\Cc$ to be abelian, we shall require it to be merely \emph{normal}. By a regular category $\Cc$ we mean a finitely complete category with the property that any arrow in $\Cc$ factorises as a regular epimorphism followed by a monomorphism, and these factorisations are pullback-stable.  
 
\begin{definition} \cite{ZJ}
A regular pointed category $\Cc$ is called \emph{normal} if every regular epimorphism in $\Cc$ is normal.
\end{definition} 
Any semi-abelian \cite{JMT} (hence in particular any abelian) category is normal and, more generally, any homological category in the sense of \cite{BB}. Examples of normal categories will be considered in the next section. A property of normal categories we shall be needing is that pullbacks reflect monomorphisms: in a pullback square
$$\xymatrix{
P \ar[r] \ar[d]_{p^*(f)} & A \ar[d]^f \\
E \ar[r]_{p} & B 
}$$
 the arrow $f$ is a monomorphism whenever $p^*(f)$ is a monomorphism \cite{BournZJanelidze}.\\
 Another interesting property of normal categories is the following:
 
\begin{lemma}\label{thirdIso}
Let $\mathbb C$ be a normal category. Then, given two normal subobjects $k \colon K \rightarrow A$ and $l \colon L \rightarrow A$ such that $K \subseteq L$, i.e. $k$ factors through $l$, then there is an isomorphism $$A/L \cong \frac{A/K}{L/K}. $$
\end{lemma}
\begin{proof}
Consider the following commutative diagram of short exact sequences:
$$
\xymatrix@=15pt{
& & 0 \ar[d] & & \\
0 \ar[r] & K \ar@{=}[d]  \ar[r] & L \ar[r] \ar[d]^l & L/K \ar@{.>}[d]^{\lambda}
 \ar[r] & 0 \\
0 \ar[r] & K \ar[r]^k & A \ar[r]^-{\pi_K} \ar[d] & A/K \ar@{.>}[d]^{\pi}  \ar[r] & 0  \\ 
  & & A/L \ar[d] \ar@{=}[r]& A/L \\
  & & 0& &  }
$$
The fact that the arrow $\pi_K \circ l \colon L \rightarrow A/K$ in $\mathbb C$ factors as a normal epimorphism followed by a monomorphism implies that the dotted arrow $\lambda$, induced by the universal property of the quotient $L \rightarrow L/K$, is a monomorphism. Furthermore, the induced arrow $\pi \colon A/K \rightarrow A/L$ is a normal epimorphism, as is the 
induced factorisation $\phi \colon L \rightarrow \mathsf{Ker}(\pi)$ such that $\mathsf{ker}(\pi) \circ \phi= \pi_K \circ l$ (since normal epimorphisms are pullback-stable).
It then follows that the right-hand vertical sequence is short exact, since $\pi$ is the cokernel of its kernel $\lambda$.
\end{proof}

As mentioned in the introduction, we shall not have to require $(\Yc,\Xc)$ to be hereditary. However, as in \cite{CJKP}, we \emph{will} be needing the condition that $\Xc$ ``covers'' $\Cc$, but we now have to make a choice: in an abelian category, an epimorphism is the same as a normal epimorphism and also the same as an effective descent morphism, but in a general normal category these notions are all distinct. Here, we shall require  $\Xc$  to have the following condition:  for every object $B\in\Cc$, we shall assume the existence of an \emph{effective descent morphism} $E\to B$ such that $E\in\Xc$.

The following condition, ``invisible'' in the abelian context, will also be needed: 
 \begin{enumerate}
 \item[(N)]\label{conditionPageN} for any normal monomorphism $k \colon K \rightarrow A$, the monomorphism \\ $k \circ t_K \colon T(K) \rightarrow A$ is normal.
\end{enumerate}
(Here we have written $t_K\colon T(K)\to K$ for the counit of the coreflection $\Cc\to \Yc$.)

Before stating and proving our main result, we recall some results from \cite{EGinf} needed in the proof, as well as the following notation: we write $\overline{\Ef}$ for the class of all normal epimorphisms $f\colon A\to B$ whose kernel $\Ker(f)$ is in $\Yc$, and $\overline{\Mf}$ for the class of all morphisms $f\colon A\to B$ whose kernel $\Ker(f)$ is in $\Xc$.

The first result is Proposition $3.5$ in \cite{EGinf}:
\begin{proposition}\label{inducedfactorisation}
If $(\Yc,\Xc)$ is a torsion theory in a normal category $\Cc$ satisfying condition $(N)$, then $(\overline{\Ef},\overline{\Mf})$ is a stable factorisation system on $\Cc$.
\end{proposition}
\begin{proof}
Let us recall how to construct the $(\overline{\Ef}, \overline{\Mf})$-factorisation of an arrow $f \colon A \rightarrow B$ in $\Cc$. If $k \colon K \rightarrow A$ is the kernel of $f$ we know, by assumption, that the monomorphism $k \circ t_K \colon T(K) \rightarrow A$ is normal. By factorising $f$ through the quotient $A/T(K)$ one gets the $(\overline{\Ef}, \overline{\Mf})$-factorization $m \circ q$ of $f$:
$$
\xymatrix{A \ar[r]^-q & A/T(K) \ar[r]^-m & B.}
$$
It is obvious that $q$ belongs to $\overline{\Ef}$; the fact that $m$ is in $\overline{\Mf}$ follows from Lemma \ref{thirdIso}, from which we can deduce that the kernel of $m$ is given by $K/T(K)=HI(K)$. Observe that the class $\overline{\Ef}$ is pullback-stable, since normal epimorphisms are pullback-stable by assumption, and the kernels of two parallel arrows in a pullback square are isomorphic. The 
rest of the proof is easy, and it can be found in \cite{EGinf}.
\end{proof}
The following Lemma was also proved in \cite{EGinf}: we recall its proof for the reader's convenience.
\begin{lemma}\label{barstable}
Given any torsion theory $(\Yc,\Xc)$ in a pointed category with pullback-stable normal epimorphisms, one always has that $\overline{\Ef}\subseteq \Ef'$.
\end{lemma}
\begin{proof}
 Since $\overline{\Ec}$ is pullback-stable, to prove that $\overline{\Ef}\subseteq \Ef'$ it suffices to show that $\overline{\Ef}\subseteq \Ef$. Consider a normal epimorphism $f$ in $\Cc$ with $\mathsf{Ker}(f) \in\Yc$. Then $f$ is the cokernel of its kernel $\ker (f)$ and, consequently, $I(f)$ the cokernel (in $\Xc$) of $I(\ker(f))$. Since $I(\mathsf{Ker}(f))=0$ by assumption, this implies that $I(f)$ is an isomorphism.
\end{proof}

%
We can now prove the main result of this article:
\begin{theorem}\label{main}
Let $\Cc$ be a normal category, $\Xc$ a torsion-free subcategory of $\Cc$ satisfying condition (N), and $(\Ef,\Mf)$ the associated reflective factorisation system. Consider the following list of properties:
\begin{enumerate}
\item
$\Xc$ ``covers'' $\Cc$: for any object $B\in\Cc$, there exists an effective descent morphism $E\to B$ such that $E\in\Xc$;
\item
$\Mf^*=\overline{\Mf}$;
\item
$\Ef'=\overline{\Ef}$;
\item
$(\Ef',\Mf^*)$ is a factorisation system.
\end{enumerate}
The implications $(1)\Rightarrow (2)\Rightarrow (3)$ and $(2)\Rightarrow (4)$ hold. $(2)\Rightarrow (1)$ holds as soon as $\Cc$ has enough projectives (with respect to effective descent morphisms).
\end{theorem}
\begin{proof}
$(1)\Rightarrow (2)$. It is well known and easily verified that $\Mf^*\subseteq \overline{\Mf}$. For the other inclusion, consider a morphism $f\colon A\to B$, an effective descent morphism $p\colon E\to B$ and the diagram
\[\vcenter{\begin{equation}\label{coveringdiagram}
\xymatrix{
HI(P) \ar[d]_{HI(p^*(f))} & P \ar@{}[rd]|<<{\pullback} \ar[d]_{p^*(f)}\ar[l]_-{\eta_P} \ar[r] & A \ar[d]^f\\
HI(E) & E \ar[r]_p \ar[l]^-{\eta_E} & B,}
\end{equation}}
\]
and assume that both $\Ker(f)$ and $E$ are in $\Xc$. Since in this case both  $\Ker(p^*(f))\cong \Ker(f)$ and the regular image $\Im[p^*(f)]$ of $p^*(f)$ lie in $\Xc$ --- the latter because it is a subobject of $E$ --- it follows that also $P\in\Xc$. Indeed, the torsion-free subcategory $\Xc$ is closed in $\Cc$ under extensions: 
\[
\xymatrix{
0 \ar[r] & \Ker(p^*(f)) \ar[r] & P \ar[r] & \Im(p^*(f)) \ar[r] & 0.}
\]
Consequently, $p^*(f)$ lies in $\Xc$. In particular, $p^*(f)\in\Mf$, hence $f\in\Mf^*$.

$(2)\Rightarrow (1)$. If $\Cc$ has enough projectives then (1) is equivalent to the condition that every projective lies in $\Xc$. We verify the latter. For this, first of all note that if $p\colon E\to B$ is an effective descent morphism such that $E$ is projective with respect to effective descent morphisms, then any $f\colon A\to B$ in $\Mf^*$ is necessarily \emph{split} by $p$, which means that $p^*(f)$ is a trivial covering, i.e. a member of the class $\Mf$: this follows from the pullback-stability of $\Mf$. 

Now let $P$ be a projective object. Then $0\to P$ is in $\Mf^*= \overline{\Mf}$, and the identity $1_P\colon P\to P$ an effective descent morphism with projective domain. It follows that $0\to P$ is split by $1_P$. Consequently, $0\to P$ is in $\Mf$, and this implies that $P\in\Xc$, since pullbacks reflect monomorphisms in the normal category $\Cc$, and the arrow $\eta_P \colon P \rightarrow HI(P)$ is then both a normal epimorphism and a monomorphism. This proves that every projective lies in $\Xc$, as desired.

$(2)\Rightarrow (3)$. $\overline{\Ef}\subseteq\Ef'$ by Lemma \ref{barstable}. If (2) holds, then since $\Ef'\subseteq (\Mf^*)^{\uparrow}$ (by Proposition 6.7 in \cite{CJKP}) and by Lemma \ref{inducedfactorisation} we have that $\Ef'\subseteq (\Mf^*)^{\uparrow}=\overline{\Mf}^{\uparrow}=\overline{\Ef}$. 

$(2)\Rightarrow (4)$ follows immediately from Proposition \ref{inducedfactorisation} and $(2)\Rightarrow (3)$.
\end{proof}
\begin{remark}
Observe that $(4)$ is strictly weaker than conditions $(1)$, $(2)$ and $(3)$:  if $(\Yc,\Xc)=(\Cc,\{0\})$ is the trivial torsion theory in a (non-trivial) normal category $\Cc$, then we have that $\Ec'=\Ec$ is the class of all morphisms, $\overline{\Ec}$ the class of the normal epimorphisms, and $\Mf^*= \Mf$ the class of all isomorphisms. Therefore condition $(4)$ is satisfied, while $(3)$ is not.
\end{remark}

\section{Examples}

\subsection{Torsion theories in the category of abelian groups}
For any torsion theory $(\Yc,\Xc)$ in an abelian category $\mathbb C$ the condition $(N)$ is trivially satisfied, so that in order to apply Theorem \ref{main}, it suffices to know that $\Xc$ ``covers'' $\Cc$. The case of the variety $\Cc=\Ab$ of abelian groups is particularly simple: here \emph{every} torsion-free subcategory $\Xc$ satisfies the latter condition, with the exception of the trivial $\Xc=\{0\}$. Indeed, any free abelian group necessarily belongs to every non-trivial torsion-free subcategory, by the following argument (which we adapted from the proof of Proposition $5.5$ in \cite{RT}):
 assuming that there exists a free abelian group $F$ that is not in $\Xc$, for a given torsion theory $(\Yc,\Xc)$ in $\Cc$, we consider the induced canonical short exact sequence
$$
\xymatrix{ 0 \ar[r] & T(F) \ar[r] & F \ar[r]  & HI(F) \ar[r] & 0}
$$
and observe that the subgroup $T(F) \not=0$ of $F$ is free, as is any subgroup of a free abelian group. Hence we have a non-trivial free abelian group in the torsion subcategory $\Yc$. This implies, first of all, that the free abelian group on the one-element set is also in $\Yc$ (since $\Yc$ is closed in $\mathsf{Ab}$ under quotients) and then that the same is true for an arbitrary free abelian group (since $\Yc$ is closed in $\mathsf{Ab}$ under coproducts) so that, finally, \emph{any} abelian group must be in $\Yc$ (once again since $\Yc$ is closed in $\mathsf{Ab}$ under quotients). Hence, $(\Yc,\Xc)=(\Ab,\{0\})$ is the only torsion theory in $\Ab$ which does not have all free abelian groups in $\Xc$. For every other torsion theory in $\Ab$ we therefore obtain (via Theorem \ref{main}) a monotone-light factorisation system $(\Ef',\Mf^*)$ where $\Ef'$ consists of the surjective group homomorphisms with kernel in $\Yc$, and $\Mf^*$ consists of the homomorphisms with kernel in $\Xc$. 

For instance, we might take for $\Xc$ (respectively, $\Yc$) the full subcategory of $\Ab$ of torsion-free (respectively, torsion) abelian groups in the usual sense. Or, $\Xc=\mathsf{Red}$ could be the full subcategory of $\Ab$ of all reduced groups and $\Yc=\mathsf{Div}$ the one of all divisible groups. Note that the latter torsion theory is not hereditary (see, for instance, \cite{Bo2}). Moreover, since any abelian group admits a monomorphism into a divisible one, the ``dual'' torsion theory $(\mathsf{Red}^{\textrm{op}},\mathsf{Div}^{\textrm{op}})$ in $\Ab^{\textrm{op}}$ (which again is not hereditary) also satisfies condition $(1)$ in Theorem \ref{main} and therefore gives us a monotone-light factorisation system in the category $\Ab^{\textrm{op}}$, which is known to be equivalent to the category of compact abelian groups.

\subsection{Torsion theories in the category of groups}
For any torsion theory  $(\Yc,\Xc)$ in the (semi-abelian, hence normal) variety $\Gp$ of groups the functoriality of the radical $T$ implies that the subgroup $T(G)$ of any group $G$ is necessarily characteristic. Accordingly, condition $(N)$ is always satisfied: indeed, given a normal monomorphism $K \rightarrow A$ in $\Gp$, $T(K)$ is a characteristic subgroup of $K$, thus also a normal subgroup of $A$.
On the other hand, just as in the case of $\Ab$, if $(\Yc,\Xc)$ is a non-trivial torsion theory in $\Gp$, then there exists for every group $G$ a surjective homomorphism $p \colon X \rightarrow G$ with $X$ in $\Xc$ \cite{Shm}. This can be shown using the same argument as in the previous example: since every subgroup of a free group is free, every non-trivial torsion-free subcategory $\Xc$ must contain all free groups.  Accordingly, we may apply Theorem \ref{main} to any non-trivial torsion theory $(\Yc,\Xc)$ in $\Gp$. For instance, we could take for $\Xc$ the subcategory of torsion-free groups in the usual sense, and for $\Yc$ the subcategory of groups generated by their elements of finite order. 

\subsection{ Hausdorff groups versus totally disconnected groups.}
Let $\GrH$ be the (homological \cite{BC} hence, in particular, normal) category of Hausdorff groups, and $\GrTot$ its full subcategory of totally disconnected groups. The forgetful functor $H \colon \GrTot \rightarrow \GrH$ admits a left adjoint $I \colon \GrH \rightarrow \GrTot$, which sends a Hausdorff group $G$ to the quotient $I(G) = G / c_0(G)$ of $G$ by the connected component  $c_0(G)$ of the neutral element $0$ of $G$.
The category $ \GrTot$ is known to be a torsion-free subcategory of $\GrH$ (see \cite{BG, GR}, for instance). Furthermore, condition $(N)$ is easily seen to be satisfied: indeed, by the same argument as in the previous example, we have for any normal subgroup $K \rightarrow A$ in $\GrH$, that $T(K) =  c_0(K)$ is a normal subgroup of $A$. Since, moreover, the topology of $c_0(K)$ is the one induced by that of $A$, the inclusion $c_0(K) \rightarrow A$ is a normal monomorphism in $\GrH$.
Finally, as was proved by Arkhangel'skii in \cite{Ark}, any Hausdorff group is a regular quotient of a totally disconnected group, so that condition $(1)$ in Theorem \ref{main} is indeed satisfied. Hence, we obtain on $\GrH$ the following monotone-light factorisation system $(\Ef',\Mf^*)$: the morphisms in $\Ef'$ are the open surjective homomorphisms in $\GrH$ with a connected kernel, whereas the class $\Mf^*$ consists of the continuous homomorphisms with a totally disconnected kernel. Since for  a continuous homomorphism $f\colon A\to B$ any of the fibres  $f^{-1}(b)$ ($b\in B$) is connected (respectively, totally disconnected) as soon as the kernel $\Ker(f)=f^{-1}(0)$ is connected (respectively, totally disconnected), the morphisms in $\Ef'$ (respectively, in $\Mf^*$) are thus precisely the continuous maps which are monotone (respectively, light) in the usual sense. Hence, Theorem \ref{main} shows, in particular, that every continous homomorphism admits a monotone-light factorisation in the classical sense, even if the spaces $A$ and $B$ are not required to be compact.

It would be interesting to know whether this result can be extended to arbitrary categories of Hausdorff semi-abelian algebras in the sense of \cite{BC}. The obstruction comes from the possibility of extending to the semi-abelian context the result which states that every Hausdorff group is a regular quotient of a totally disconnected group, as already observed in \cite{GJ}.

\subsection{Commutative rings versus reduced rings}\label{rings}
Recall that a commutative ring $A$ (not necessarily with unit) is \emph{reduced} if it has no (non-zero) nilpotent element: for all $a \in A$ and $n\ge 1$, $a^n=0$
implies that $a=0$. We write $\mathsf{CRng}$ for the (semi-abelian, hence normal) variety of commutative rings, $\mathsf{RedCRng}$ for the quasivariety of reduced rings, and $H \colon \mathsf{RedCRng} \rightarrow \mathsf{CRng}$ for the forgetful functor. As observed in \cite{CDT}, $\mathsf{RedCRng}$ is a torsion-free subcategory of $\mathsf{CRng}$, which is hereditary since the corresponding torsion subcategory $\mathsf{NilCRng}$ of nilpotent commutative rings is closed in $\mathsf{CRng}$ under subrings. This implies that the reflector $I \colon \mathsf{CRng} \rightarrow \mathsf{RedCRng}$ preserves monomorphisms, so that condition $(N)$ is satisfied: indeed, given a normal monomorphism (= ideal) $k \colon {K} \rightarrow  {A}$ in $\mathsf{CRng}$, we obtain a commutative diagram of short exact sequences 
$$
\xymatrix{0 \ar[r] & T({K}) \ar[r]^{t_K} \ar[d]_{T(k)} & {K} \ar[r]^-{\eta_{{K}}} \ar[d]^k & HI( {K}) \ar[r]  \ar[d]^{HI(k)}& 0 \\
0 \ar[r] & T( {A}) \ar[r]_{t_A} &  {A} \ar[r]_-{\eta_{{A}}} & HI( {A}) \ar[r] & 0
}
$$
in $\mathsf{CRng}$, where $HI(k)$ is a monomorphism. Accordingly, the left hand square is a pullback, and $k \circ t_K \colon T(K) \rightarrow A$ a normal monomorphism, as an intersection of normal monomorphism. It is also clear that any free commutative ring is reduced, 
so that Theorem \ref{main} applies to the torsion theory $(\mathsf{NilCRng}, \mathsf{RedCRng})$ in $\mathsf{CRng}$. The induced  factorisation system $(\Ef',\Mf^*)$ is the following: a ring homomorphism in $\mathsf{CRng}$ is in $\Ef'$ if it is surjective with a nilpotent kernel, and it belongs to $\Mf^*$ when its kernel is reduced.

\subsection{Internal groupoids versus equivalence relations}
Let $\mathbb C$ be a normal Mal'tsev category, and let $\mathsf{Grpd}(\mathbb C)$ be the category of (internal) groupoids in $\Cc$. Recall that a groupoid $\mathsf{A}=(A_1,A_0,m, d,c,i)$ 
  in $\Cc$ is a diagram of the form 
\[
\xymatrix{
A_1\times_{A_0}A_1 \ar[r]^-{m} & A_1 \ar@<1.3 ex>[rr]^{d} \ar@<-1.8 ex>[rr]_{c} && A_0, \ar@<0.7 ex>[ll]_{i}}
\]
where $A_0$ represents the ``object of objects'', $A_1$ the ``object of arrows'', $A_1\times_{A_0}A_1$ the ``object of composable arrows'', $d$ the ``domain'', $c$ the ``codomain'', $i$ the ``identity'', and $m$ the ``composition''. Of course, these morphisms are required to satisfy the commutativity conditions expressing, internally, the fact that $\mathsf{A}$ is a groupoid. 
The category $\mathsf{Grpd}(\mathbb C)$ is regular and Mal'tsev \cite{Gran}, and it is also a normal category. This latter observation easily follows from the fact that a morphism $(f_0,f_1) \colon \mathsf{A} =(A_1,A_0,m, d,c,i) \rightarrow \mathsf{B} = (B_1,B_0,m, d,c,i)$ is a normal epimorphism in $\mathsf{Grpd}(\mathbb C)$ if and only if both the arrows $f_0 \colon A_0 \rightarrow B_0$ and $f_1 \colon A_1 \rightarrow B_1$ are normal epimorphisms in $\mathbb C$.
The category $\mathsf{Eq}(\mathbb C)$ of internal equivalence relations in $\mathbb C$ is a torsion-free subcategory of $\mathsf{Grpd}(\mathbb C)$, as explained in Example 5.5 in \cite{BG} (in the more restrictive context of homological categories \cite{BB}, but the arguments used there still apply in the present context). The corresponding torsion part consists of the subcategory $\mathsf{Ab}(\mathbb C)$ of abelian objects in $\mathbb C$ (seen as particular internal groupoids). The fact that $\mathsf{Eq}(\mathbb C)$ ``covers'' $\mathsf{Grpd}(\mathbb C)$ follows from the well known fact that any groupoid $\mathsf{A}=(A_1,A_0,m, d,c,i)$ in $\mathbb C$ is a regular quotient of the equivalence relation $\mathsf{Eq(}d\mathsf{)} = (A_1 \times_{A_0} A_1, A_1, \tau, p_1, p_2, (1_{A_1}, 1_{A_1}))$ which occurs as the kernel pair of $d$. The fact that condition $(N)$ holds true follows from the fact that the reflector $I \colon \mathsf{Grpd}(\mathbb C) \rightarrow \mathsf{Eq}(\mathbb C)$ to the torsion-free subcategory $\mathsf{Eq}(\mathbb C)$ clearly preserves monomorphisms, so that the argument recalled in the Example \ref{rings} also applies to the present example.

Remark that the category $\mathsf{Grpd}(\mathsf{Grp})$ of internal groupoids in the category $\Gp$ of groups is known to be equivalent to the category $\mathsf{XMod}$ of crossed modules introduced by J.H.C. Whitehead \cite{W}. An object in $\mathsf{XMod}$ is a group homomorphism $\alpha \colon A \rightarrow B$ together with an action of $B$ on $A$, written ${\,}^{b} a$ for any $a \in A$ and $b \in B$, such that $\alpha({\,}^{b} a)= b a b^{-1}$ and ${\,}^{\alpha(a)} a_1= a a_1 a^{-1}$ for any $a, a_1 \in A$, $b \in B$. An arrow $(f_0, f_1) \colon \alpha \rightarrow \alpha'$ in the category $\mathsf{XMod}$ is a pair of group homomorphisms making the diagram
\begin{equation}\label{arrow}
\vcenter{\xymatrix{A \ar[r]^{f_1} \ar[d]_{\alpha} & A' \ar[d]^{\alpha'} \\
B \ar[r]_{f_0} & B'
}}
\end{equation}
commute, and preserving the action: $f_1 ( {\,}^{b} a) = {\,}^{f_0(b)} f_1(a)$.

The equivalence $\mathsf{Grpd}(\mathsf{Grp})\cong \mathsf{XMod}$ restricts to an equivalence between the category $\mathsf{Eq}(\Gp)$ of internal equivalence relations in $\Gp$ and the category $\mathsf{NormMono}$ of normal monomorphisms of groups, where the action is given by conjugation. The reflector $I \colon \mathsf{XMod} \rightarrow \mathsf{NormMono}$ sends a crossed module $\alpha \colon A \rightarrow B$ to the normal monomorphism $\alpha(A) \rightarrow B$ which is the inclusion of the image $\alpha(A)$ in $B$. The monotone-light factorisation of an arrow $(f_0,f_1) \colon \alpha \rightarrow \alpha'$ in $ \mathsf{XMod}$ is then obtained as follows. One considers the kernel $(\ker(f_0),\ker(f_1)) \colon \hat{\alpha} \rightarrow \alpha$ of $(f_0,f_1)$ in $ \mathsf{XMod}$, and one then factors $A$ by its normal subgroup $\mathsf{Ker}(\hat{\alpha})$, so that $(f_0,f_1) \colon \alpha \rightarrow \alpha'$ in (\ref{arrow}) decomposes into the commutative diagram
$$
\xymatrix{A \ar[r]^-{\pi_{ \mathsf{Ker}(\hat{\alpha})}} \ar[d]_{\alpha} & A/ \mathsf{Ker}(\hat{\alpha}) \ar@{.>}[r]^-{\phi_1} \ar[d] & A' \ar[d]^{\alpha'} \\
B \ar@{=}[r]_{1_B} & B \ar[r]_{f_0} & B',
}
$$
where $\phi_1 \circ \pi_{ \mathsf{Ker}(\hat{\alpha})} = f_1$,  $(1_B, \pi_{ \mathsf{Ker}(\hat{\alpha})})$ is in $\Ef'$, and $(f_0, \phi_1)$ is in $\Mf^*$.

\end{document}